\documentclass[10pt, reqno]{amsart}
\usepackage{amsmath, amsthm, amscd, amsfonts, amssymb, graphicx, color}
\usepackage[bookmarksnumbered, colorlinks, plainpages]{hyperref}
\hypersetup{colorlinks=true,linkcolor=red, anchorcolor=green, citecolor=cyan, urlcolor=red, filecolor=magenta, pdftoolbar=true}
\usepackage{mathrsfs}

\newtheorem{theorem}{Theorem}

\theoremstyle{definition}

\newtheorem{example}[theorem]{Example}
\newtheorem{tables}[theorem]{Table}

\theoremstyle{remark}
\newtheorem{remark}[theorem]{Remark}

\begin{document}
\setcounter{page}{1}

\title[Hilbert 12th problem]{Addendum to ``Measured foliations and Hilbert 12th problem''}

\author[Nikolaev]
{Igor V. Nikolaev$^1$}

\address{$^{1}$ Department of Mathematics and Computer Science, St.~John's University, 8000 Utopia Parkway,  
New York,  NY 11439, United States.}
\email{\textcolor[rgb]{0.00,0.00,0.84}{igor.v.nikolaev@gmail.com}}

\dedicatory{All data are available as part of the manuscript}

\subjclass[2010]{Primary 11G45; Secondary 46L85.}

\keywords{class field theory,  noncommutaive geometry}


\begin{abstract}
We study numerical examples of the abelian extensions of the real quadratic number
fields based on  the results  in \textit{Acta Mathematica Vietnamica}  {\bf 48} (2023), 271-281
(arXiv:0804.0057) 
\end{abstract}

\maketitle

\section{Introduction}
An explicit class field theory  for the real  quadratic number fields $\mathbf{Q}(\sqrt{D})$
dates back to the doctoral thesis  [Hecke 1910] \cite{Hec1}.  Advised  by D.~Hilbert, 
the author introduced the field  $\mathbf{Q}(\sqrt{D}, \sqrt{-1})$ and
an analog of the $j$-invariant depending on two complex variables. 
In Hecke's thesis,  the automorphic functions became intrinsically and permanently
intertwined with the class field theory.  This  idea was revisited 
in  [Hecke 1928] \cite{Hec2}.  Let us recall some notation and trivia.

For an integer $N>1$   consider the principal congruence subgroup 
 $\Gamma(N)=\left(\begin{smallmatrix} 1 & 0\cr 0 & 1\end{smallmatrix}\right)  \mod N$
of the modular group $SL_2(\mathbf{Z})$. 
Let $\mathbb{H}=\{z=x+iy\in {\Bbb C} ~|~ y>0\}$ be the upper half-plane  and 
let $\Gamma(N)$  act on $\mathbb{H}$  by the linear fractional
transformations;  consider an orbifold  $\mathbb{H}/\Gamma(N)$.
To compactify the orbifold 
at the cusps, one adds a boundary to $\mathbb{H}$,  so that 
$\mathbb{H}^*=\mathbb{H}\cup \mathbf{Q}\cup\{\infty\}$;  the compact Riemann surface 
$X(N)=\mathbb{H}^*/\Gamma(N)$ is called  a  modular curve.   
The meromorphic functions $f(z)$ on $\mathbb{H}$ that
vanish at the cusps and such that
$f\left({az+b\over cz+d}\right)=(cz+d)^2f(z),
~\forall \left(\begin{smallmatrix} a & b\cr c & d\end{smallmatrix}\right)\in\Gamma(N)$, 
are  known as  cusp forms of weight $2$;  the (complex linear) space of such forms
is denoted by $S_2(\Gamma(N))$.  The formula $f(z)\mapsto \omega=f(z)dz$ 
defines an isomorphism  $S_2(\Gamma(N))\cong \Omega_{hol}(X(N))$, where 
$\Omega_{hol}(X(N))$ is the space of holomorphic differentials
on the Riemann surface $X(N)$.  In particular, 
$\dim_{\mathbf{C}}(S_2(\Gamma(N))=\dim_{\mathbf{C}}(\Omega_{hol}(X(N))=g$,
where $g=g(N)$ is the genus of the surface $X(N)$. 
The Hecke operator,  $T_n$, acts on $S_2(\Gamma(N))$ by the formula
$T_n f=\sum_{m\in \mathbf{Z}}\gamma(m)q^m$, where
$\gamma(m)= \sum_{a|\hbox{\bf{GCD}}(m,n)} a c_{mn/a^2}$ and 
$f(z)=\sum_{m\in \mathbf{Z}}c(m)q^m$ is the Fourier
series of the cusp form $f$ at $q=e^{2\pi iz}$.  Further,  $T_n$ is a
self-adjoint linear operator on the vector space $S_2(\Gamma(N))$
endowed with the Petersson inner product;  the group ring
$\mathbb{T}_N :=\mathbf{Z}[T_1,T_2,\dots]$ is a commutative algebra.
Any cusp form $f\in S_2(\Gamma(N))$ that is an eigenvector
for one (and hence all) of $T_n$, is referred  to
as a Hecke eigenform.
Let $f\in S_2(\Gamma(N))$ be a (normalized) Hecke eigenform, such that
\linebreak
$f(z)=\sum_{n=1}^{\infty}c_n(f)q^n$ its Fourier series. We shall denote
by $K_f=\mathbf{Q}(\{c_n(f)\})$ the algebraic number field generated
by the Fourier coefficients of $f$.
It is well known that $1\le deg~(K_f~|~\mathbf{Q})\le 2g$
[Diamond \& Shurman 2005]   \cite[pp. 234-235]{DS}.
Finally, let $\mathscr{H}(k)$ 
($\mathscr{H}(k)\mod\mathfrak{f}$, resp.) 
be the Hilbert class field (ring class field modulo conductor $\mathfrak{f}\ge 1$, resp.)  
of a number field $k$.  A relation between the  Fourier coefficients 
of the Hecke eigenforms  and the  class field theory of imaginary 
quadratic fields is given by the following fundamental result. 
\begin{theorem} \label{thm1}
{\bf (\cite[p. 709]{Hec2})} 
For each prime $p\in\{4n+3 ~|~n\ge 1\}$ there exists a Hecke eigenform $f\in S_2(\Gamma (p))$, 
such that:
\begin{equation}\label{eq1.1}
\mathscr{H}(\mathbb{Q}(\sqrt{-p}))\subseteq K_f. 
\end{equation}
\end{theorem}
\begin{remark}
The  inclusion (\ref{eq1.1}) extends to the Pythagorean primes $p\in\{4n+1 ~|~n\ge 1\}$ 
[Shimura 1972] \cite[Tables $I_a$ and  $I_b$]{Shi1} and
(in a modified form) to the composite levels $N$ [Hecke 1928] \cite{Hec2}. 
A rigorous proof of these inclusions is unknown to the author.  
\end{remark}
Yu.~I.~Manin came to realize that the analytic methods  are insufficient for
 an explicit class field theory of the real quadratic number fields  (Hilbert 12th problem).  
The missing tools should involve ideas  from noncommutative geometry, such as
pseudo-lattices with real multiplication [Manin 2003] \cite{Man1}. 
Roughly speaking, the latter are measured foliations on the Riemann surfaces.  
Manin's idea proved to be  visionary.
\begin{theorem} \label{thm3}
{\bf (\cite[Theorem 1.1]{Nik2})}
For each level $N>1$ there exists a real quadratic number field $\mathbf{Q}(\sqrt{D})$
and a conductor $\mathfrak{f}\ge 1$, 
such that $\mathscr{H}(\mathbf{Q}(\sqrt{D}))\mod \mathfrak{f}\subseteq K_f$.  
\end{theorem}
Notice that a perfect solution to Hilbert's 12th problem must be a converse of Theorem \ref{thm3}, i.e. one needs 
to calculate $N$ as  a function  of the discriminant $D$.  This was settled in \cite[Theorem 1]{Nik1} by a general formula $N=\mathfrak{f}'D$,
where  $\mathfrak{f}'$ is an integer number depending on the conductor  $\mathfrak{f}\ge 1$. 
The aim of this Addendum is a precise formula based on Hecke's Theorem \ref{thm1}. 
\begin{theorem}\label{thm4}
For each prime $p\in\{4n+3 ~|~n\ge 1\}$ there exists an integer $\mathfrak{f}_1\ge 1$, 
such that $\mathscr{H}(\mathbf{Q}(\sqrt{p}))\mod\mathfrak{f}_1\subseteq K_f$,
where  $f\in \{S_2(\Gamma (m^2p)) ~|~m\ge 1\}$ is a  Hecke eigenform. 
\end{theorem}
\begin{proof}
Let  $p\in\{4n+3 ~|~n\ge 1\}$ be a prime. Recall  \cite[formula (2)]{Nik1} that there exists the least integers
$\mathfrak{f}_1,  \mathfrak{f}_2\ge 1$ satisfying an isomorphism of the class groups: 
\begin{equation}\label{eq2}
Cl~(\mathbf{Q}(\sqrt{p}) \mod\mathfrak{f}_1) \cong Cl~(\mathbf{Q}(\sqrt{-p})\mod\mathfrak{f}_2).
\end{equation}
Since  $Gal~(\mathscr{H}(\mathbb{Q}(\sqrt{-p})) \mod \mathfrak{f}_2)\cong  Cl~(\mathbf{Q}(\sqrt{-p})\mod\mathfrak{f}_2)$,
the action of the group $Cl~(\mathbf{Q}(\sqrt{p}) \mod\mathfrak{f}_1)$ extends to the field $K_f$ (Theorem \ref{thm1}). 
In other words, the ring class field   $\mathscr{H}(\mathbf{Q}(\sqrt{p}))\mod\mathfrak{f}_1$  belongs to $K_f$. 
It remains to notice,  that inclusion (\ref{eq1.1}) is also true  for the  levels $\{m^2p ~|~m\ge 1\}$.  
\end{proof}

\section{Experimental data}
In this section the numerical examples illustrating Theorem \ref{thm4} are compiled.  We used
the  online databases
\textbf{lmfdb.org},  \textbf{numbertheory.org}, PARI/GP software
 \textbf{pari.math.u-bordeaux.fr} and Magma calculator  \textbf{magma.maths.usyd.edu.au}. 
 To  clarify notation in  Table \ref{tb7} below, let us consider the simplest example of the prime $p=7$.  
\begin{example}
 Consider the prime $p=7\in\{4n+3 ~|~n\ge 1\}$ for $n=1$. 
 The class groups $Cl~(\mathbf{Q}(\sqrt{-7}))$ and   $Cl~(\mathbf{Q}(\sqrt{7}))$
 are  trivial.  This case corresponds to the  conductors 
 $\mathfrak{f}_1 = \mathfrak{f}_2 = 1$. The least non-trivial conductors 
 satisfying (\ref{eq2}) are $\mathfrak{f}_1=3$ and   $\mathfrak{f}_2=4$
 with
\begin{equation}\label{eq3}
Cl~(\mathbf{Q}(\sqrt{7})\mod 3)\cong Cl~(\mathbf{Q}(\sqrt{-7}) \mod 4) \cong \mathbf{Z}/2\mathbf{Z}. 
\end{equation}
 To calculate the ring class field $\mathscr{H}(\mathbf{Q}(\sqrt{7})\mod 3)$,
 we let $N=63=3^27$. We are looking for a Hecke eigenform $f\in S_2(\Gamma(63))$ with complex multiplication
 by $\sqrt{-7}$ and $deg~(K_f|\mathbf{Q})=4$.  Such a form exists and has the coefficient 
 field $K_f$ given by the minimal polynomial $x^4+8x^2+9$ (\textbf{lmfdb.org}). 
 The roots of the latter are:
\begin{equation}\label{eq4}
x_{1,2,3,4}=\pm i\sqrt{4\pm \sqrt{7}}. 
\end{equation}
  A passage to the measured foliation $\mathcal{F}=Re~f$  \cite[p. 273]{Nik2} 
  is equivalent to taking imaginary component of the complex numbers. 
  Therefore the roots (\ref{eq4}) 
  become  $\pm \sqrt{4\pm \sqrt{7}}$ having  the  minimal
  polynomial $x^4-8x^2+9$. 
  We conclude that: 
\begin{equation}\label{eq5}
 \mathscr{H}\left(\mathbf{Q}(\sqrt{7})\mod 3\right)\cong \mathbf{Q}\left(\pm \sqrt{4\pm \sqrt{7}}\right).  
  \end{equation}
\end{example}

\begin{tables}\label{tb7}
\end{tables}

\begin{table}[h]
\begin{tabular}{c|c|c|c|c|c|c|c}
\hline
$p=$& Class & Class &  & & Ring  & $N=$ & Minimal\\
$4n+3,$ & group & group & $\mathfrak{f}_1$ &  $\mathfrak{f}_2$ &class& $m^2p,$  &polynomial of\\
$n\ge 1$&  $\mathbf{Q}(\sqrt{p})$  &  $\mathbf{Q}(\sqrt{-p})$ &&&group&  $m\ge 1$ & $\mathscr{H}\left(\mathbf{Q}(\sqrt{p})\right) ~mod ~\mathfrak{f}_1$\\
\hline
$7$ & trivial & trivial  & $3$ & $4$ & $\mathbf{Z}/2\mathbf{Z}$ &  $3^27$   &$x^4-8x^2+9$\\
\hline
$7$ & trivial & trivial  & $5$ & $3$ & $\mathbf{Z}/4\mathbf{Z}$ &  $5^27$   &$x^8 +\dots+1$\\
\hline
$11$ & trivial & trivial  & $4$ & $3$ & $\mathbf{Z}/2\mathbf{Z}$ & $3^211$ & $x^4+\dots+9$\\
\hline
$19$ & trivial & trivial  & $5$ & $3$ & $\mathbf{Z}/4\mathbf{Z}$ & $6^219$& $x^8+\dots+1350$\\
\hline
$23$ & trivial &   $\mathbf{Z}/3\mathbf{Z}$ & $7$ & $3$ & $\mathbf{Z}/6\mathbf{Z}$ & $3^223$ & $x^{12}+\dots+64$\\
\hline
$31$ & trivial &  $\mathbf{Z}/3\mathbf{Z}$   & $9$  & $4$ &   $\mathbf{Z}/6\mathbf{Z}$  & $3^231$&  $x^{12}+\dots+225$ \\
\hline
$43$ & trivial & trivial  & $5$ & $3$ &  $\mathbf{Z}/4\mathbf{Z}$ & $6^243$ & $x^8+\dots+21222$\\
\hline
$47$ & trivial &   $\mathbf{Z}/5\mathbf{Z}$ & $11$ & $3$ &   $\mathbf{Z}/10\mathbf{Z}$ & $3^247$ & $x^{20}+\dots+1024$\\
\hline
$59$ & trivial &  $\mathbf{Z}/3\mathbf{Z}$  & $7$ & $3$ &  $\mathbf{Z}/6\mathbf{Z}$ & $3^259$ & $x^{12}+\dots+729$\\
\hline
$67$ & trivial & trivial  & $5$ & $3$ &  $\mathbf{Z}/4\mathbf{Z}$ & $6^267$& $x^8+\dots+106326$\\
\hline
$71$ & trivial &  $\mathbf{Z}/7\mathbf{Z}$  & $>50$ & $>10$ & order $>50$  & $m>10$ & degree $>100$\\
\hline
$79$ & $\mathbf{Z}/3\mathbf{Z}$  & $\mathbf{Z}/5\mathbf{Z}$   & $>50$ & $>10$ & order $>50$ & $m>10$ & degree $>100$\\
\hline
$83$ & trivial & $\mathbf{Z}/3\mathbf{Z}$ & $7$  & $3$ &  $\mathbf{Z}/6\mathbf{Z}$ & $3^283$ & $x^{12}+\dots+729$\\
\hline
$103$ & trivial & $\mathbf{Z}/5\mathbf{Z}$ & $11$  & $4$ &  $\mathbf{Z}/10\mathbf{Z}$ & $3^2103$ & $x^{20}+\dots+1521$\\
\hline
$107$ & trivial & $\mathbf{Z}/3\mathbf{Z}$ & $7$  & $3$ &  $\mathbf{Z}/6\mathbf{Z}$ & $3^2107$ & $x^{12}+\dots+1089$\\
\hline
$127$ & trivial & $\mathbf{Z}/5\mathbf{Z}$ & $11$  & $4$ &  $\mathbf{Z}/10\mathbf{Z}$ & $3^2127$ & $x^{20}+\dots+3969$\\
\hline
$131$ & trivial & $\mathbf{Z}/5\mathbf{Z}$ & $50$  & $5$ &  $\mathbf{Z}_2\oplus \mathbf{Z}_{20}$ & $m>10$ & $x^{80}+\dots$\\
\hline
$139$ & trivial & $\mathbf{Z}/3\mathbf{Z}$ & $13$  & $3$ &  $\mathbf{Z}/12\mathbf{Z}$  & $6^2139$ & $x^{24}+\dots$\\
\hline
$151$ & trivial & $\mathbf{Z}/7\mathbf{Z}$ & $29$  & $3$ &  $\mathbf{Z}/28\mathbf{Z}$    & $m>10$ & $x^{56}+\dots$\\
\hline
$163$ & trivial & trivial & $8$  & $3$ &  $\mathbf{Z}/4\mathbf{Z}$  & $6^2163$ & $x^{8}+\dots+3120822$\\
\hline

\end{tabular}
\end{table}


\begin{remark}
The values for the entries $p=71,79, 131,151$ exceed current capacity of \textbf{lmfdb.org} and/or Magma.  
\end{remark}

\section*{Data availability}
  
  Data sharing not applicable to this article as no datasets were generated or analyzed during the current study.
   
\section*{Conflict of interest}
On behalf of all co-authors, the corresponding author states that there is no conflict of interest.
  

\bibliographystyle{amsplain}


\end{document}